\documentclass[12pt,leqno,a4paper]{amsart}
\usepackage{amssymb}
\usepackage{xcolor}

\newcommand{\FF}{{\mathbb{F}}}
\newcommand{\QQ}{{\mathbb{Q}}}
\newcommand{\RR}{{\mathbb{R}}}
\newcommand{\ZZ}{{\mathbb{Z}}}

\newcommand{\bB}{{\mathbf{B}}}
\newcommand{\bG}{{\mathbf{G}}}
\newcommand{\bP}{{\mathbf{P}}}
\newcommand{\bT}{{\mathbf{T}}}

\newcommand{\cF}{{\mathcal{F}}}
\newcommand{\cH}{{\mathcal{H}}}

\newcommand{\fS}{{\mathfrak{S}}}

\newcommand{\rX}{{\mathrm{X}}}
\newcommand{\Ind}{{\operatorname{Ind}}}
\newcommand{\Res}{{\operatorname{Res}}}
\newcommand{\Irr}{{\operatorname{Irr}}}
\newcommand{\rk}{{\operatorname{rk}}}
\newcommand{\sgn}{{\operatorname{sgn}}}
\newcommand{\St}{{\operatorname{St}}}
\newcommand{\SL}{{\operatorname{SL}}}
\newcommand{\SU}{{\operatorname{SU}}}
\newcommand{\SO}{{\operatorname{SO}}}
\newcommand{\Sp}{{\operatorname{Sp}}}
\newcommand{\LR}{{\operatorname{LR}}}
\newcommand{\uni}{{\operatorname{uni}}}

\newcommand{\tw}[1]{{}^#1\!}
\newcommand{\Chevie}{{\sf Chevie}{}}

\let\la=\lambda

\newtheorem{thm}{Theorem}[section]
\newtheorem{lem}[thm]{Lemma}
\newtheorem{prop}[thm]{Proposition}
\newtheorem{conj}[thm]{Conjecture}

\theoremstyle{remark}
\newtheorem{rem}[thm]{Remark}
\newtheorem{exmp}[thm]{Example}

\begin{document}

\title[A positivity conjecture]{A positivity conjecture for the\\ Alvis-Curtis
  dual of the intersection cohomology\\ of a Deligne--Lusztig variety}

\date{\today}

\author{Olivier Dudas}
\address{Universit\'e Paris Diderot, UFR de Math\'ematiques,
  B\^atiment Sophie Germain, 5 rue Thomas Mann, 75205 Paris CEDEX 13, France.}
\email{dudas@math.jussieu.fr}

\author{Gunter Malle}
\address{FB Mathematik, TU Kaiserslautern, Postfach 3049,
         67653 Kaisers\-lautern, Germany.}
\email{malle@mathematik.uni-kl.de}

\thanks{The second author gratefully acknowledges financial support by ERC
  Advanced Grant 291512.}



\begin{abstract}
We formulate a strong positivity conjecture on characters
afforded by the Alvis-Curtis dual of the intersection cohomology
of Deligne--Lusztig varieties. This conjecture provides a powerful
tool to determine decomposition numbers of unipotent $\ell$-blocks
of finite reductive groups.
\end{abstract}

\maketitle

%%%%%%%%%%%%%%%%%%%%%%%%%%%%%%%%%%%%%%%%%%%%%%%%%%%%%%%%%%%%%%%%%%%%%%%%%
\section{Introduction}  \label{sec:intro}

An important part of the modular representation theory of a finite group is
encoded in its decomposition matrix, which relates simple representations in
characteristic zero with those in positive characteristic. Determining such
matrices amounts to finding the characters of the indecomposable projective
representations.
\par
In the case of finite reductive groups, projective modules can be constructed
by Harish-Chandra induction. However, like in the ordinary case, many modules
do not appear in the induction from proper Levi subgroups. They correspond to
the projective covers of the so-called cuspidal modules. For these, one needs
to use Deligne--Lusztig induction instead of Harish-Chandra induction, the
problem being that the latter produces virtual modules. We conjecture that one
can obtain proper projective modules, and so overcome this problem, by
considering suitable linear combinations of these virtual modules.
\par
Let us give more details about the construction. Let $\bG$ be a connected
reductive algebraic group over an algebraically closed field of positive
characteristic with a Steinberg endomorphism $F$ making $\bG^F$ into a finite
reductive group. To any element $w$ of the Weyl group of $\bG$
one can associate a quasi-projective variety $\rX(w)$ acted on by $\bG^F$, the
Deligne--Lusztig variety. The alternating sum of the $\ell$-adic cohomology
groups of $\rX(w)$ yields a virtual character $R_w$ of $\bG^F$. Taking instead
the Alvis-Curtis dual of the intersection cohomology one obtains another
virtual character $Q_w$, which is a linear combination of the Deligne--Lusztig
characters $R_y$. The coefficients of this combination are the values at $1$
of the twisted Kazhdan--Lusztig polynomials, which also express simple objects
in terms of Verma modules in the principal block of the category
$\mathcal{O}$ of a semi-simple Lie algebra. Surprisingly, unlike the
Deligne--Lusztig characters the $Q_w$'s are no longer virtual. The purpose
of this note is to conjecture a modular analogue of that property and
to give evidence towards it (see \S\ref{sec:conjecture}).

\begin{conj}
  Assume that $\ell$ is not too small. Then up to a sign, $Q_w$ is
  the unipotent part of a projective character.
\end{conj}

We shall explain the strong implications of this conjecture to the
determination of decomposition matrices for finite reductive groups, and
give several instructive examples where this conjecture holds.
We believe that a general proof might rely on the geometric realization
of the Alvis-Curtis duality.

%%%%%%%%%%%%%%%%%%%%%%%%%%%%%%%%%%%%%%%%%%%%%%%%%%%%%%%%%%%%%%%%%%%%%%%%%
\section{Statement of the conjecture}  \label{sec:DL}

\subsection{Deligne--Lusztig theory}

Let $\bG$ be a connected reductive linear algebraic group over an algebraically
closed field of positive characteristic $p$, and $F:\bG\rightarrow\bG$ be a
Steinberg endomorphism. There exists a positive integer $\delta$ such that
$F^\delta$ defines a split $\FF_{q^\delta}$-structure on $\bG$ (with
$q \in \RR_+$), and we will choose $\delta$ minimal for this property.
We set $G:=\bG^F$, the finite group of fixed points.

We fix a pair $(\bT,\bB)$ consisting of a maximal torus contained in a Borel
subgroup of $\bG$, both of which are assumed to be $F$-stable. We denote by $W$
the Weyl group of $\bG$, and by $S$ the set of simple reflections in $W$
associated with $\bB$. Then $F^\delta$ acts trivially on $W$.
Following Lusztig \cite[17.2]{Lu86}, for any $F$-stable irreducible character
$\chi$ of $W$, one can choose a preferred extension $\widetilde \chi$ of
$\chi$ to the group $W \rtimes \langle F \rangle$ satisfying
$\widetilde \chi (x F^\delta) = \widetilde \chi(x)$ for all
$x \in W \rtimes \langle F \rangle$.
The almost character associated to $\widetilde \chi$ is then the
following uniform character
$$ R_{\widetilde \chi}
      = \frac{1}{|W|} \sum_{w\in W} \widetilde\chi(wF) R_{\bT_{wF}}^\bG (1).$$
Let $\rX(w)$ denote the Deligne--Lusztig variety associated with $w$. Then
its $\ell$-adic cohomology groups $H^i(\rX(w))$ with coefficients in
$K \supset \QQ_\ell$ give rise to the so-called \emph{unipotent}
representations of $G$.
By \cite[Thm.~4.23]{Lu84} the decomposition of almost characters in terms of
unipotent characters can be computed explicitly. Conversely, for $w \in W$ the
orthogonality relations for the Deligne--Lusztig characters yield
$$R_{\bT_{wF}}^\bG (1) = \sum_{i\in\ZZ} (-1)^i \big[H^i(\rX(w))\big]
      = \sum_{\chi \in (\Irr\, W)^F} \widetilde\chi(wF) R_{\widetilde\chi}$$
as a virtual $KG$-module.
\par
The Frobenius $F^\delta$ acts on the Deligne--Lusztig variety $\rX(w)$, making
the cohomology groups $H^i(\rX(w))$ into
$G\times \langle F^\delta\rangle$-modules. Digne and Michel \cite{DM85} have
extended the previous formula to take into account this action.
By \cite[Cor.~3.9]{Lu78}, the eigenvalues of $F^\delta$ on a unipotent
character $\rho$ in the cohomology of $\rX(w)$ are of the form $\lambda_\rho
q^{\delta m/2}$ where $\lambda_\rho$ is a root of unity which depends only on
$\rho$ and $m$ is a nonnegative integer. We fix an indeterminate $v$
and we shall denote by $v^{\delta m} \rho$ the class in the Grothendieck group
of $G\times \langle F^\delta \rangle$-modules of such a representation.
\par
Let $\cH_v(W)$ be the Iwahori--Hecke algebra of $W$ with equal parameters $v$.
By convention, the standard basis $(t_w)_{w\in W}$ of $\cH_v(W)$ will satisfy
the relation $(t_s + v)(t_s - v^{-1})=0$ for all $s \in S$. For
$\chi\in(\Irr W)^F$ we denote by
$\widetilde \chi_v$ the character of $\cH_v(W) \rtimes \langle F \rangle$
which specializes to $\widetilde \chi$ at $v =1$. We denote by $(C_w')_{w\in W}$
(resp.~$(C_w)_{w\in W}$) the Kazhdan--Lusztig basis (resp.~twisted
Kazhdan--Lusztig basis) of $\cH_v(W)$. For any simple reflection $s\in S$ we
have $C_s' = t_s+v$ and $C_s = t_s - v^{-1}$.
If $\iota$ denotes the $\ZZ[v,v^{-1}]$-linear involution on
$\cH_v(W)$ defined by $\iota(t_w) = (-1)^{\ell(w)} t_{w^{-1}}^{-1}$
then $\iota (C_w') = (-1)^{\ell(w)} C_w$.
The virtual $G\times \langle F^\delta \rangle$-modules afforded by
the intersection cohomology of Deligne--Lusztig varieties can
be computed by means of the $C_w'$ base. For the following result, see
\cite[Thm.~3.8]{Lu84}:

\begin{thm}[Lusztig]   \label{thm:dlvar}
 Let $w \in W$. The class in the Grothendieck group of
 $G\times \langle F^\delta\rangle$-modules of the intersection cohomology
 of $\rX(w)$ is given by
 $$\sum_{i \in\ZZ} (-1)^i \big[IH^i\big(\overline{\rX(w)}\big)\big]
      = (-1)^{\ell(w)}\sum_{\chi \in (\Irr\, W)^F}
      \widetilde \chi_v(C_w' F) R_{\widetilde \chi}.$$
\end{thm}

\begin{exmp}
The group $G = \SL_2(q)$ has two unipotent characters: the trivial character
$1_G$ and the Steinberg character $\St_G$. Here, the element $t_s$ represents
the cohomology of $\rX(s)$
whereas the element $C_s' = t_s + v$ represents the cohomology of
$\overline{\rX(s)} \cong \bP_1$. The irreducible characters
of $\cH_v(W)$ corresponding to the trivial and sign characters of $W$ are
defined by $1_v(t_s ) = v^{-1}$ and $\sgn_v(t_s) = -v$. We have
$[H^0(\rX(s))] = 1_G$ and $[H^1(\rX(s))] = v^2 \St_G$ so that
$$[H^0(\rX(s))] - [H^1(\rX(s))] = 1_G - v^2 \St_G.$$
It corresponds to the element $v t_s$ since $1_v(vt_s) = 1$ and
$\sgn_v(vt_s) = -v^2$.

We have also $[H^0\big(\overline{\rX(s)}\big)] = 1_G$,
$[H^1\big(\overline{\rX(s)}\big)] = 0$ and
$[H^2\big(\overline{\rX(s)}\big)] = v^2 1_G$ so that
$$[H^0\big(\overline{\rX(s)}\big)] - [H^1\big(\overline{\rX(s)}\big)] +
[H^2\big(\overline{\rX(s)}\big)] = 1_G + v^2 1_G$$
corresponds to $v C_s'$ since $1_v(vC_s') = v(v^{-1} + v) = 1+v^2$ and
$\sgn_v(vC_s') = v(-v +v) = 0$. Since $\overline{\rX(s)}$ is smooth and
one-dimensional, $IH^{i}\big(\overline{\rX(s)}\big) =
H^{i+1}\big(\overline{\rX(s)}\big)(-1)$, where $(-1)$ is a Tate twist
(contributing $v^{-1}$ in the character). Consequently the intersection
cohomology of $\overline{\rX(s)}$ corresponds to $-C_s'$.
\end{exmp}

\subsection{Basic sets for finite reductive groups}  \label{se:setting}
Let $\ell$ be a prime number and $(K,\mathcal{O},k)$ be an $\ell$-modular
system. We assume that it is large enough for all the finite groups
encountered. Furthermore, since we will be working with $\ell$-adic cohomology
we will assume throughout this note that $K$ is a finite extension of
$\QQ_\ell$.
\par
Let $H$ be a finite group. Representations of $H$ will always be assumed
to be finite-dimensional. Recall that every projective $kH$-module lifts to
a representation of $H$ over $K$. The character afforded by such a
representation will be referred to as a \emph{projective character}. Integral
linear combinations of projective characters will be called \emph{virtual
projective characters}.
\par

\medskip
\setlength\fboxsep{0pt}
\colorbox{gray!20}{\fbox{\parbox{15cm}{Throughout this paper, we shall
 make the following assumptions on $\ell$:
 \begin{itemize}
  \item $\ell \neq p$ (non-defining characteristic),
  \item $\ell$ is good for $\bG$ and $\ell \nmid |(Z(\bG)/Z(\bG)^\circ)^F|$.
 \end{itemize}}}
}
\medskip

In this situation, the unipotent characters lying in a given unipotent
$\ell$-block of $G$ form a \emph{basic set} of this block \cite{GH91,Ge93}.
Consequently, the restriction of the
decomposition matrix of the block to the unipotent characters is invertible.
In particular every (virtual) unipotent character is a virtual projective
character, up to adding and removing some non-unipotent characters.

\subsection{A positivity conjecture \label{sec:conjecture}}
Let $D_G$ denote the Alvis--Curtis duality, with the convention that if $\rho$
is a cuspidal unipotent character, then $D_G(\rho) = (-1)^{\rk_F (\bG)} \rho$
where $\rk_F (\bG) = |S/F|$ is the $F$-semisimple rank of $\bG$. Then
$D_G(R_{\widetilde \chi}) = R_{\widetilde \chi \otimes \widetilde \sgn}$
(see \cite[6.8.6]{Lu84}) and
$(\widetilde \chi \otimes \widetilde{\mathrm{sgn}})_v  =
\widetilde \chi_v \circ \iota$
(see for example \cite[5.11.4]{Lu84}) so that by Theorem~\ref{thm:dlvar}
$$\sum_{i \in\ZZ} (-1)^i \big[D_G\big(IH^i(\overline{\rX(w)})\big)\big]
  = \sum_{\chi \in (\Irr\, W)^F} \widetilde \chi_v(C_w F) R_{\widetilde \chi}.$$
We denote by $Q_w$ the restriction to $G$ of this virtual character, and by
$Q_w[\lambda]$ the generalized $\lambda$-eigenspace of $F^\delta$ for
$\lambda \in k^\times$.
\par
Lusztig proved \cite[Prop.~6.9 and~6.10]{Lu84} that up to a global sign,
$Q_w$ (and even $Q_w[\lambda]$) is always a nonnegative combination
of unipotent characters (note that the assumption on $q$ can be
removed by \cite[Cor.~3.3.22]{DMR}). The sign is given by the $a$-value $a(w)$
of the two-sided cell in which $w$ lies.

\begin{prop}[Lusztig] \label{prop:cwchar0}
  For all $w \in W$ and all $\la \in k^\times$, $(-1)^{a(w)} Q_w[\la]$
  is a sum of unipotent characters.
\end{prop}

Unipotent characters are only the unipotent part of virtual
projective characters in general. We conjecture that the modular analogue
of Proposition~\ref{prop:cwchar0} should hold in general, that is that $Q_w$ is
actually a proper projective character whenever
$\ell$ is not too small.

\begin{conj}   \label{conj:cwproj}
 Under the assumption on $\ell$ in \S\ref{se:setting},
 for all $w \in W$ and all $\lambda \in k^\times$,
 $(-1)^{a(w)}Q_w[\lambda]$ is the unipotent part of a projective character.
\end{conj}

\begin{exmp}
The closure of the Deligne--Lusztig variety $\rX(w_0)$ associated with the
longest element $w_0$ of $W$ is smooth and equal to $\bG/\bB$. Therefore its
intersection
cohomology $IH^\bullet(\overline{\rX(w_0)})$ consists of copies of the trivial
representation in degrees of a given parity. Since $a(w_0) = \ell(w_0)$,
then $(-1)^{a(w_0)}Q_{w_0}$
is a nonnegative multiple of the Steinberg character. On the other hand,
the Steinberg character is the unipotent part of a unique projective
indecomposable module (given by a summand of a Gelfand--Graev module) and
therefore the conjecture holds for $w_0$.
\end{exmp}

\begin{rem}
This conjecture has been checked on many decomposition matrices, including
the unipotent blocks with cyclic defect groups for exceptional groups,
and the matrices that are determined in \cite{Du13,DM13}. We
will give some examples in \S~\ref{se:examples}.
\end{rem}

%%%%%%%%%%%%%%%%%%%%%%%%%%%%%%%%%%%%%%%%%%%%%%%%%%%%%%%%%%%%%%%%%%%%%%%%%%%%%%%
\section{Applications}

In this section we explain how to deduce properties
of decomposition matrices using Conjecture~\ref{conj:cwproj}.

\subsection{Families of simple unipotent modules} Following \cite[\S5]{Lu84},
we denote by $\leq_\LR$ the partial order on two-sided cells of $W$. Recall
that to each cell $\Gamma$ corresponds a two-sided ideal $I_{\leq \Gamma}
:= \mathrm{span}_{\QQ}\{{C_w}_{|v=1}\mid w\leq_\LR\Gamma\}$
of the group algebra $\QQ W$.  Moreover, given $\chi \in \Irr\, W$, there is
a unique two-sided cell $\Gamma_\chi$ such that $\chi$ occurs in
$I_{\leq \Gamma_\chi}/I_{< \Gamma_\chi}$. To each two-sided cell $\Gamma$
one can attach a so-called family $\cF_\Gamma$ of unipotent characters. They
are defined as the constituents of $R_\chi$ for various $\chi$ such that
$\Gamma_\chi = \Gamma$. By \cite[Thm.~6.17]{Lu84}, they form a partition of the
set of unipotent characters of $G$.
\par
We can use the characters $Q_w$ to define families of simple unipotent
$kG$-modules. For each $Q_w$ we choose a virtual projective $kG$-module
$\widetilde Q_w$ whose character coincides with $Q_w$ up to adding/removing
non-unipotent characters. We denote by $\Irr_{\leq \Gamma} kG$
(resp.~$\Irr_{< \Gamma} kG$) the set of simple unipotent $kG$-modules $N$ such
that a  projective cover $P_N$ occurs in the virtual module $\widetilde Q_w$
for some $w\leq_\LR \Gamma$ (resp.~$w <_\LR \Gamma$).
We define the family of simple unipotent modules associated to $\Gamma$
as $\Irr_{\Gamma} kG =\Irr_{\leq \Gamma} kG \setminus \Irr_{< \Gamma} kG$.
Since the regular representation is uniform (see for example
\cite[Cor.~12.14]{DM91}) then every indecomposable projective module
lying in a unipotent block occurs in some $\widetilde Q_w$,
therefore every simple unipotent module belongs to at least one family.
However it is unclear whether this family is unique in general.

\begin{prop}
 Let $b$ be a unipotent $\ell$-block of $G$. Assume that
 \begin{itemize}
  \item[(i)] every unipotent character in $b$ is a linear combination of
   $bR_\chi$'s, and
  \item[(ii)] Conjecture~\ref{conj:cwproj} holds.
 \end{itemize}
 Then for every two-sided cell $\Gamma$
 $$|\Irr_{\Gamma} bkG | = |\cF_\Gamma \cap \Irr\, b|.$$
 In particular, any simple $b$-module lies in a unique family.
\end{prop}

\begin{proof}
Given $\chi \in \Irr\,W$ and $\Gamma_\chi$ the corresponding two-sided cell,
the primitive idempotent $e_\chi \in \QQ W$ lies in $I_{\leq \Gamma_\chi}$.
In other words, it is a linear combination of ${C_w}_{|v=1}$'s with
$w \leq_{\LR} \Gamma_\chi$. Consequently, for all $\chi \in \Irr\,W$
$$ R_\chi \in \mathrm{span}\{Q_w\mid w\leq_\LR \Gamma_\chi\}.$$
Therefore $|\Irr_{\leq \Gamma} bkG| \geq \dim \mathrm{span} \{bR_\chi\mid\Gamma_\chi \leq \Gamma\}$
which is the number of unipotent characters in $b$ whose family is smaller
than $\Gamma$. But every $bQ_w$ is a linear combination of these characters,
and so is every PIM occurring in $b\widetilde Q_w$ if
Conjecture~\ref{conj:cwproj} holds. In particular, $|\Irr_{\leq \Gamma} bkG|$
has to be smaller than the number of these characters and we conclude using
the assumption (i).
\end{proof}

\begin{rem}
Note that (i) is not always satisfied, for example when two complex conjugate
characters lie in the same block. However the validity of this assumption
is easy to check on the Fourier matrices. For example, it is valid whenever
$\ell | (q\pm1)$ and $\bG$ is exceptional.
\end{rem}

\subsection{Application to the decomposition matrix}
It is conjectured that, for $\ell$ not too small, the $\ell$-modular
decomposition matrix of $G$
depends only on the order of $q$ modulo $\ell$. The following proposition
gives some evidence toward this conjecture as well as to Geck's conjecture
on the unitriangular shape of the decomposition matrix
\cite[Conj.~3.4]{GH97}.

\begin{prop}
 Let $b$ be a unipotent $\ell$-block of $G$. Assume that
 \begin{itemize}
  \item[(i)] Conjecture~\ref{conj:cwproj} holds,
  \item[(ii)] $|\Irr_{\Gamma} bkG | = |\cF_\Gamma \cap \Irr\, b|$ for any
   two-sided cell $\Gamma$ of $W$.
 \end{itemize}
 Then the unipotent part of the decomposition matrix of $b$ has the following
 shape:
 $$D_\uni \, = \, \left( \begin{array}{cccc} D_{\cF_1} & 0 & \cdots & 0\\
            {}* & D_{\cF_2} & \ddots & \vdots \\
            \vdots & \ddots & \ddots & 0 \\
            {}* & \cdots  & * & D_{\cF_r} \\
            \end{array}\right)$$
 where $\cF_i$ runs over the families and where each  $D_{\cF_i}$ is a square
 matrix of size $|\cF_i \cap \Irr\,b|$.  Furthermore, the entries of $D_\uni$
 are bounded above independently of $q$ and $\ell$.
\end{prop}

\begin{proof}
Given any simple $kG$-module $N$ in the block, there exists by (ii) a unique
two-sided cell $\Gamma$ such that $N \in \Irr_{\Gamma} bkG$  and $P_N$
occurs in some $Q_w$ for $w \in \Gamma$ (note that $w {\not<}_\LR \Gamma$
otherwise $N$ would belong to a smaller family). Since the matrix
of the $Q_w$'s is block triangular with respect to families, the proposition
follows from (i).
\end{proof}

\subsection{Determining decomposition numbers} The bounds on the entries
of the decomposition matrix given by the $Q_w$'s are often small enough
to determine some of the decomposition numbers. We illustrate this on
the principal $\Phi_2$-block of the group $\mathrm{F}_4(q)$.

\begin{prop}
 Assume that $(q,6)=1$, $q\equiv-1\pmod\ell$ and $\ell > 11$. If
 Conjecture~\ref{conj:cwproj} holds, then the following matrix
 $$\begin{array}{c|ccccc}
    \phi_{9,10}		& 1\\
    \phi_{2,16}''	& 1& 1\\
    \phi_{2,16}'		& 1& .& 1\\
    B_2,\varepsilon	& 2& .& .& 1\\
    \phi_{1,24}		& 3& 2& 2& 4& 1\\
 \end{array}$$
 is a submatrix of the $\ell$-modular decomposition matrix of
 $\mathrm{F}_4(q)$. (Here, the "$\cdot$"s denote zero entries.)
\end{prop}

\begin{proof}
By \cite{Koe06}, there exist integers $f,g,h,i,j$ with
$f,h,i \geq 2$, $g\geq 4f-5$ and $j \geq 4$ such that the matrix
  $$\begin{array}{c|ccccc}
    \phi_{9,10}		& 1\\
    \phi_{2,16}''	& 1& 1\\
    \phi_{2,16}'		& 1& .& 1\\
    B_2,\varepsilon	& f& .& .& 1\\
    \phi_{1,24}		& g& h& i& j& 1\\
  \end{array}$$
is a submatrix of the decomposition matrix of $\mathrm{F}_4(q)$.
Let $b$ be the block idempotent associated with the principal
$\ell$-block of $\mathrm{F}_4(q)$. With
$w = s_2s_3s_2s_3s_4s_3s_2s_1s_3s_2s_4s_3s_2s_1$ we have
$$bQ_w = 32 \phi_{9,10} + 72\phi_{2,16}'' + 72 \phi_{2,16}'
         +72 B_2,\varepsilon  +288 \phi_{1,24}.$$
If we denote by $P_1,\ldots,P_5$ the unipotent part of the characters of the
PIMs corresponding to the last five columns of the decomposition matrix of
$\mathrm{F}_4(q)$, we can decompose $Q_w$ as
$$bQ_w = 32P_1 + 40P_2 + 40P_3 + (72-32f)P_4+(288-32g-40h-40i-(70-32f)j)P_5.$$
By Conjecture~\ref{conj:cwproj} the multiplicity of each $P_k$ should be
nonnegative, that is
$72-32f \geq 0$ and $288-32g-40h-40i-(70-32f)j \geq 0$. Since $f\geq 2$
the first relation forces $f=2$ (and therefore $g \geq 3$). The second becomes
$288-32g-40h-40i-8j \geq 0$.
 Since $288 = 32\time 3 + 40 \times 2 + 40\times 2
+8\times 4$ is the minimal value that the expression $32g+40h+40i+8j$ can
take, we deduce that $h=i=2$, $g=3$ and $j=4$.
\end{proof}

%%%%%%%%%%%%%%%%%%%%%%%%%%%%%%%%%%%%%%%%%%%%%%%%%%%%%%%%%%%%%%%%%%%%%%%%%%%%%%%
\section{Some evidence}   \label{se:examples}

\subsection{A cuspidal module in the unitary group}
We give here a non-trivial example for $\SU_n(q)$ where
Conjecture~\ref{conj:cwproj} holds.
The key point is to find a formula for $(C_w F)_{|v=1}$ in terms of
well-identified elements of $W$. This is done using the geometric description
of Kazhdan--Lusztig polynomials. The proof given here can be adapted to other
groups, even when the Schubert variety is no longer smooth, using
Bott--Samelson varieties instead.
\par
Recall that the set of unipotent characters of $\SU_n(q)$ is parametrized
by partitions of $n$. Given such a partition $\lambda$, we denote
by $\rho_\la$ (resp. $\chi_\la$) the corresponding unipotent character
(resp. character of $\fS_n$), with the convention that $\rho_{1^n}$ is
the Steinberg character.

\begin{prop}   \label{prop:subreg}
 Let $\ell > n$ be a prime dividing $q+1$. Then
 Conjecture~\ref{conj:cwproj} holds for $\SU_n(q)$ and $w = s_1 w_0 \in \fS_n$.
 Furthermore,
 $$\frac{(-1)^{a(w)}}{(n-1)!} Q_{w} = \rho_{21^{n-2}} +(n-1)\rho_{1^n}$$
 is the unipotent part of a projective indecomposable $k\SU_n(q)$-module.
\end{prop}

Here $w_0$ is the longest element of $W \cong \fS_n$ and $s_i$
is the transposition $(i,i+1)$. For computing $Q_w$ for $w=s_1w_0$
we first need to compute the decomposition of the corresponding
Kazhdan--Lusztig element on the standard basis:

\begin{lem}   \label{lem:cw}
 If $w_0$ is the longest element of $W$ and $w_I$ the longest element of
 $W_I$ with $I = \{s_1,\ldots,s_{n-2}\}$ we have
 $$C_{s_1 w_0}'
   =  v^{-1} C_{w_0}' - v^{n-2} t_{s_1 s_2 \cdots s_{n-1}}C_{w_I}'.$$
\end{lem}

\begin{proof}[Proof of the lemma]
Let $\bB$ be a Borel subgroup of $\bG=\SL_n$ and $\bP$ be the standard
parabolic subgroup corresponding to the set of simple reflections
$I = \{s_1,s_2,\ldots,s_{n-2}\}$. Let $\pi : \bG/\bB \longrightarrow \bG/\bP$
be the canonical projection. The variety $\bG/\bP$ is isomorphic to the
projective space $\bP_{n-1}$ and is paved by the affine spaces
$\bB s_i s_{i+1} \ldots s_{n-1} \bP /\bP$ of dimension $n-i$.
The closure of each of them is in turn a projective space of
dimension $n-i$, and hence is smooth.

\smallskip

Each element $w \in W$ can be written uniquely as
$w = s_i s_{i+1} \ldots s_{n-1} x$ with $1 \leq i \leq n$ and $x \in W_I$.
In particular, the image of the corresponding Schubert cell under $\pi$ is
exactly $\bB s_i s_{i+1} \ldots s_{n-1} \bP / \bP$. We deduce that
$\overline{\bB s_1 w_0 \bB / \bB}
  = \pi^{-1} \big(\overline{ \bB s_2 s_{3} \ldots s_{n-1} \bP / \bP}\big)$.
In particular, it is smooth and by \cite[Thm.~A2]{KL79} the Kazhdan--Lusztig
element is given by
$$ C_{s_1 w_0}' = \sum_{w \leq s_1 w_0} v^{\ell(s_1 w_0)-\ell(w)} t_w
  = v^{-1} \Big(\sum_{w \in W} v^{\ell(w_0)-\ell(w)} t_w
       - \sum_{w \nleq s_1 w_0} v^{\ell(w_0)-\ell(w)} t_w\Big).$$
Now $w \nleq s_1 w_0$ if and only if $w = s_1 s_2 \cdots s_{n-1} x$ with
$x \in W_I$. In that case one can write $t_w = t_{s_1 \cdots s_{n-1}} t_x$ and
the result follows from the relation $\ell(w)-\ell(w_I) = n-1$ and the
expression for $C_v'$ when $v$ is the longest element of a parabolic subgroup
(which again comes from the smoothness of $\overline{\bB w_I \bB/\bB}
= \bP_I/\bB$).
\end{proof}

\begin{proof}[Proof of Proposition~\ref{prop:subreg}]
By definition,
$$Q_w = \sum_{i\in\ZZ}(-1)^i\big[D_G\big(IH^i(\overline{\rX(w)})\big)\big]_{|v=1}
      = \sum_{\chi \in (\Irr\, W)^F} \widetilde \chi_v(C_w F)_{|v=1} R_{\widetilde\chi}.$$
Now using the involution $\iota:t_w \longmapsto (-1)^{\ell(w)} t_{w^{-1}}^{-1}$
of $\cH_v(W)$ one has $\iota(C_w') = (-1)^{\ell(w)} C_w$, so that from
Lemma~\ref{lem:cw} we get
$C_{s_1 w_0} =  -v^{-1} C_{w_0} + v^{n-2} t_{s_{n-1} \cdots s_1}^{-1} C_{w_I}$.
The evaluation at $v=1$ yields
$$\begin{array}{r@{\ \, = \ \,}l}
  (C_{s_1w_0})_{|v=1} &  \displaystyle -\sum_{w\in W} (-1)^{\ell(w_0)-\ell(w)} w
   + s_1 \cdots s_{n-1} \sum_{w \in W_I} (-1)^{\ell(w_I )- \ell(w)} w \\[15pt]
  & (-1)^{\ell(w_0)+1}|W| e_\sgn + (-1)^{\ell(w_I)} |W_I| s_1 \cdots s_{n-1} e_{\sgn_I}
\end{array}$$
where $\sgn$ (resp. $\sgn_{I}$) is the sign character of $W$ (resp. $W_I$) and
$e_\chi$ denotes the central idempotent corresponding to the character $\chi$.
Recall from \cite[17.2]{Lu86}, that the extension $\widetilde \chi$ of $\chi$
satisfies $\widetilde \chi(wF) = (-1)^{a(\chi)} \chi(ww_0)$. Using the fact
that $w_0 = s_{n-1} \cdots s_1 w_{F(I)}$ we obtain
$$\begin{array}{r@{\ \, = \ \,}l}
  (-1)^{a(\chi)} \widetilde \chi_v (C_{s_1 w_0} F)_{|v=1}
  & (-1)^{\ell(w_0)+1} |W| \chi(e_{\sgn}w_0)
  + (-1)^{\ell(w_I)} |W_I| \chi\big(w_{F(I)} e_{\sgn_{F(I)}}\big)\\[8pt]
  & -|W| \chi(e_{\sgn}) +  |W_I| \chi \big(e_{\sgn_{F(I)}} \big).
\end{array}$$
Since $\chi\big(e_{\sgn_{F(I)}}) = (\Res_{W_{F(I)}}^W \chi)(e_{\sgn_{F(I)}})$,
we deduce that it is zero whenever $\chi$ is not a constituent of
$\Ind_{W_{F(I)}}^W \sgn_{F(I)}$. The two constituents of this induced
representation correspond to the partitions $1^n$ and $21^{n-2}$, with
respective $A$-functions given by $\ell(w_0)$ and $\ell(w_0)-1$. Now using
the fact that $R_{\widetilde \chi_\mu}= (-1)^{a(\mu)+A(\mu)} \rho_{\mu}$ for
any partition $\mu \vdash n$, one finds
$$Q_{s_1 w_0}
   = (-1)^{\ell(w_0)-1}|W_I| \big(\rho_{2 1^{n-2}}+ (n-1)\rho_{1^{n}}\big),$$
and $\rho_{2 1^{n-2}}+ (n-1)\rho_{1^{n}}$ is the unipotent part of the
character of an indecomposable projective module by \cite[Thm.~5.9]{DM13}.
\end{proof}

\subsection{Groups of small rank} We finish by computing for several groups $G$
of small rank the contribution to the principal $\ell$-block $b$ by various
$Q_w$'s.
In the table, we give in the second column the minimal integer $d$ such that
$\ell \mid (q^d-1)$, and in the last column the decomposition
of $Q_w$ in the basis of projective indecomposable modules which is
obtained from the decomposition matrices in \cite{HN13,Du13,DM13,HL98}.
Note however that this might change when $\ell$ is small as the decomposition
matrix will change (see for example $\mathrm{G}_2(4)$ with $\ell = 5$).
In any case, $Q_w$ remains the unipotent part of a proper projective character.
Our notation for the unipotent characters is as in the cited sources.

$$\begin{array}{c|c|c|c|c}
G 			& d & w & \la & bQ_w[\la] \\[3pt]\hline
\Sp_6(q)\vphantom{\Big)}& 2 & s_1 s_2 s_1s_2s_3& 1 & 3(\rho_{1,1^2}+\St)
				+ (B_2,1^2 + 2\St) + (\rho_{1^3,-}+3\St) \\
\textrm{G}_2(q)		& 2 & s_1s_2 & -1 & G_2[-1]+2\St\\[4pt]
\tw3\textrm{D}_4(q)	& 6 & s_1s_2s_3s_2 & 1& 2(\phi_{2,2}+\phi_{1,3}'+\St) + 2(\tw3\textrm{D}_4[-1]+2\St)\\[4pt]
\SU_6(q)			& 2& s_1 s_3 s_4 s_3& 1 & 2 ( \rho_{321} + 2 \rho_{2^3} + 2 \rho_{3 1^3}
 			+ 2 \rho_{2^21^2} + 2 \rho_{21^4} + 6 \rho_{1^6})\\[4pt]
\mathrm{F}_4(q)		&12&s_1s_2s_1s_3s_2s_1s_3s_2s_4s_3s_2s_3\!
				& \zeta_{12}^5 & 2F_4[\theta]+20B_2,\St+ 4\St\\
			&  &    & \zeta_{12}^4 & 4F_4[\theta]+F_4[\mathrm{i}]+11B_2,\St+ 10\St\\[4pt]
\end{array}$$

One can even construct larger examples for which the conjecture holds. We
give below five examples of computation of $Q_w$ in $\SU_{10}(q)$ when
$q\equiv-1\pmod\ell$. When
$\ell > 17$, we can use \cite[Thm.~6.2]{DM13} to decompose them
on the basis of projective indecomposable modules. It turns out that at least
four of them are (up to a scalar) the character of a projective
indecomposable module.
$$\begin{array}{c|c}
  w & Q_w \\\hline
  \vphantom{\Big)}
  s_1s_3s_4s_3s_5s_4s_3s_6s_5s_4s_3s_7s_6s_5s_4s_3s_9s_8s_7s_6s_5s_4s_3&
  		1440\Psi_{321^5} \\
  s_1s_2s_1s_4s_5s_4s_6s_5s_4s_7s_6s_5s_4s_8s_7s_6s_5s_4&
  		96\Psi_{32^21^3} \\[4pt]
  s_5s_4s_3s_2s_1s_4s_3s_2s_4s_3s_4s_9s_8s_7s_9s_8s_9&
  		576\Psi_{32^31} \\[4pt]
  s_1s_2s_1s_4s_5s_4s_6s_5s_4s_7s_6s_5s_4s_9&
  		96\Psi_{3^221^2} \\[4pt]
  s_2s_1s_2s_3s_7s_6s_5s_7s_6s_7s_9&
  		16\big(\Psi_{4321}+(1-\alpha)\Psi_{32^31} + \beta\Psi_{2^41^2}\big)\\
\end{array}$$
Here $\alpha,\beta \in \{0,1\}$ are as in \cite[Thm.~6.2]{DM13}

\medskip
\noindent{\bf Acknowledgement:}
We would like to thank Jean Michel for helping us to compute the characters
of the modules $Q_w$ in \Chevie\ \cite{MChv}, especially for $\SU_{10}(q)$.

%%%%%%%%%%%%%%%%%%%%%%%%%%%%%%%%%%%%%%%%%%%%%%%%%%%%%%%%%%%%%%%%%%%%%%%%%

\end{document}